\documentclass[11pt,reqno,tbtags]{amsart}
\usepackage{amssymb}
\numberwithin{equation}{section}

\usepackage{url}
\usepackage[square,numbers]{natbib}
\bibpunct[, ]{[}{]}{;}{n}{,}{,}

\makeatletter
\newtheorem*{rep@theorem}{\rep@title}
\newcommand{\newreptheorem}[2]{%
\newenvironment{rep#1}[1]{%
 \def\rep@title{#2 \ref{##1}}%
 \begin{rep@theorem}}%
 {\end{rep@theorem}}}
\makeatother

\newtheorem{theorem}{Theorem}[section]
\newreptheorem{theorem}{Theorem}
\newtheorem{lemma}[theorem]{Lemma}

\theoremstyle{definition}
\newtheorem{example}[theorem]{Example}

\newtheorem{remark}[theorem]{Remark}

\theoremstyle{remark}

\newcounter{thmenumerate}

\newcounter{xenumerate}

\newcommand\E{\operatorname{\mathbb E{}}}

\renewcommand\Pr{\operatorname{\mathbb P{}}}

\newcommand\Po{\operatorname{Po}}

\newcommand\eps{\varepsilon}

\renewcommand\phi{\varphi}

\newcommand{\sumjo}{\sum_{j=0}^\infty}
\newcommand{\sumj}{\sum_{j=1}^\infty}
\newcommand{\sumko}{\sum_{k=0}^\infty}
\newcommand{\sumk}{\sum_{k=1}^\infty}

\newcommand{\summ}{\sum_{m=1}^\infty}

\newcommand\gd{\delta}

\newcommand\gs{\sigma}

\newcommand\tM{\tilde M}  

\newcommand\set[1]{\ensuremath{\{#1\}}}
\newcommand\bigset[1]{\ensuremath{\bigl\{#1\bigr\}}}

\newcommand\bigpar[1]{\bigl(#1\bigr)}
\newcommand\Bigpar[1]{\Bigl(#1\Bigr)}

\newcommand\abs[1]{|#1|}
\newcommand\bigabs[1]{\bigl|#1\bigr|}

\newcommand\lrabs[1]{\left|#1\right|}
\def\rompar(#1){\textup(#1\textup)}    
\newcommand\xfrac[2]{#1/#2}

\def\xexp(#1){e^{#1}}
\newcommand\ceil[1]{\lceil#1\rceil}

\newcommand\ntoo{\ensuremath{{n\to\infty}}}

\newcommand\ttoo{\ensuremath{{t\to\infty}}}

\newcommand\punkt{.\spacefactor=1000}    
\newcommand\iid{i.i.d\punkt}
\newcommand\ie{i.e\punkt}
\newcommand\eg{e.g\punkt}

\newcommand\cf{cf\punkt}
\newcommand{\as}{a.s\punkt}

\newcommand\whp{w.h.p\punkt}

\newcommand{\tend}{\longrightarrow}

\newcommand\asto{\overset{\mathrm{a.s.}}{\tend}}

\newcommand\op{o_{\mathrm p}}

\newcommand{\refT}[1]{Theorem~\ref{#1}}

\newcommand{\refL}[1]{Lemma~\ref{#1}}

\newcommand{\refS}[1]{Section~\ref{#1}}

\newcounter{CC}
\newcommand{\CC}{\stepcounter{CC}\CCx} 
\newcommand{\CCx}{C_{\arabic{CC}}}     
\newcommand{\CCdef}[1]{\xdef#1{\CCx}}     
\newcounter{cc}
\newcommand{\cc}{\stepcounter{cc}\ccx} 
\newcommand{\ccx}{c_{\arabic{cc}}}     
\newcommand{\ccdef}[1]{\xdef#1{\ccx}}     

\newcommand\cB{\mathcal B}

\newcommand\cE{\mathcal E}

\newcommand\cS{\mathcal S}

\newcommand\cV{\mathcal V}

\newcommand\N{{\mathbb N}}
\newcommand\Z{{\mathbb Z}}

\newcommand\gln{\lambda_n}

\begin{document}
\title[Greedy Independent Set in a Random Graph]
{The Greedy Independent Set in a Random Graph with Given Degrees}

\date{19 October 2015} 

\author{Graham Brightwell}
\address{Department of Mathematics, London School of Economics,
Houghton Street, London WC2A 2AE, United Kingdom}
\email{g.r.brightwell@lse.ac.uk}
\urladdr{http://www.maths.lse.ac.uk/Personal/graham/}

\author{Svante Janson}
\thanks{Svante Janson is supported by the Knut and Alice Wallenberg Foundation}
\address{Department of Mathematics, Uppsala University, PO Box 480,
SE-751~06 Uppsala, Sweden}
\email{svante.janson@math.uu.se}
\urladdr{http://www.math.uu.se/{\tiny$\sim$}svante/}

\author{Malwina Luczak} \thanks{The research of Malwina Luczak is supported by an EPSRC Leadership Fellowship, grant reference EP/J004022/2.}
\address{School of Mathematics, Queen Mary, University of London}
\email{m.luczak@qmul.ac.uk}
\urladdr{http://www.maths.qmul.ac.uk/$\sim$luczak/}


\begin{abstract}
We analyse the size of an independent set in a random graph on $n$ vertices with specified vertex degrees, constructed via a simple greedy algorithm:
order the vertices arbitrarily, and, for each vertex in turn, place it in the independent set unless it is adjacent to some vertex
already chosen.  We find the limit of the expected proportion of vertices in the greedy independent set as $n \to \infty$, expressed as an integral whose upper limit is defined implicitly, valid whenever the second moment of a random vertex degree is
uniformly bounded. We further show that the random proportion of vertices in the independent set converges to the jamming constant as $n \to \infty$. The results hold under weaker assumptions in a random multigraph with given degrees constructed via the configuration model.
\end{abstract}

\maketitle

\section{Introduction}\label{sec:intro}

In this paper, we analyse a simple greedy algorithm for constructing an independent set in a random graph chosen uniformly from those with
a given degree sequence, under some mild regularity assumptions.
We obtain a nearly
explicit formula for the size of the independent set constructed.

Our method involves generating the random graph via the configuration model, simultaneously with running the greedy algorithm.  In the present
context, this idea was first used by Wormald~\cite{Wormald}, who treated the special case of random regular graphs.  Our methods follow those
developed in a sequence of papers by Janson and Luczak~\cite{JL1,JL2}, and most particularly a recent paper of Janson, Luczak and Windridge~\cite{JLW} (see the proof of Theorem 2.6(c) in that paper).

\medskip


We consider the following natural greedy process for generating an independent set $\cS$ in an $n$-vertex (multi)graph.  We start with $\cS$ empty,
and we consider the vertices one by one, in a uniformly random order.  At each step, if the vertex under consideration is not adjacent to a vertex in
$\cS$, we put it in~$\cS$, and otherwise we do nothing.  Let $S_\infty = S^{(n)}_\infty$ be the size of $\cS$ at the end of the process; the expected
value of $S_\infty/n$ is sometimes called the {\em jamming constant} of the (multi)graph.  (Note: a multigraph may have loops, and it might be
thought natural to exclude a looped vertex from the independent set, but as a matter of convenience we do not do this, and we allow looped vertices into
our independent set.
Ultimately, the main interest is in the case of
graphs.)

For $n \in \N$ and a sequence $(d_i)_1^n$ of non-negative integers, let $G(n, (d_i)_1^n)$ be a simple graph (i.e.\ with no loops or multiple edges) on $n$ vertices chosen uniformly at random from among all graphs with degree sequence $(d_i)_1^n$. (We tacitly assume there is some such graph, so $\sum_{i=1}^n d_i$ must be even, at least.) We let $G^*(n, (d_i)_1^n)$ be the random multigraph with given degree sequence $(d_i)_1^n$ defined by the configuration model, first introduced in Bollob\'as~\cite{Bollobas2}. That is, we take a set of $d_i$ half-edges for each vertex $i$ and combine the half-edges into pairs by a uniformly random matching of the set of all half-edges. Conditioned on the multigraph being a (simple) graph, we obtain $G(n, (d_i)_1^n)$, the uniformly distributed random graph with the given degree sequence.

Let
\begin{equation}
n_k = n_k (n) = \# \{i: d_i = k \}, \qquad k \in \Z^+:=\set{0,1,2,\dots},
\end{equation}
the number of vertices of degree $k$ in $G(n, (d_i)_1^n)$ (or $G^*(n, (d_i)_1^n)$). Then $\sum_k n_k = n$ and $\sum_k k n_k$ is even.

Let $(p_k)_0^{\infty}$ be a probability distribution, and assume that $n_k/n
\to p_k$ for each $k \in \Z^+$. We assume further that the distribution
$(p_k)_0^{\infty}$ has a finite and positive mean $\lambda = \sumk k p_k$
and that the average vertex degree  $\sum_k k n_k/n$ converges to
$\lambda$. (Equivalently, the distribution of the degree of a random vertex
is uniformly integrable.
In~\refS{sec:noui}, we consider the jamming constant of a random multigraph
when we relax this condition.)


We prove our results for the greedy independent set process on $G^*$, and,
by conditioning on $G^*$ being simple, we deduce that these results also
hold for the greedy independent set process on $G$. For this, we use a
standard argument that relies on the probability that $G^*$ is simple being
bounded away from zero as $n \to \infty$. By the main theorem
of~\cite{Janson}
(see also \cite{JansonII})
this occurs if and only if
$\sumk k^2 n_k(n) =O(n)$.
(Equivalently, the second moment of the degree distribution of a random
vertex is uniformly bounded.)


%

Our aim in this paper is to determine the asymptotics of the jamming
constant of the random graph $G(n, (d_i)_1^n)$ (or random multigraph
$G^*(n,(d_i)_1^n)$) in the limit
as $n \to \infty$, and to show that the size of the greedy random
independent set in $G(n, (d_i)_1^n)$ or $G^*(n,(d_i)_1^n)$ scaled by the
size of the graph converges to the jamming constant.

\begin{theorem} \label{thm.main}
Let $(p_k)_0^{\infty}$ be a probability distribution, and let $\lambda = \sumk k p_k \in (0, \infty)$. Assume that $n_k/n \to p_k$ for each $k \in \Z^+$ and that $\sumk k n_k/n \to \lambda$ as $n \to \infty$.

Let $S^{(n)}_\infty$ denote the size of a random greedy independent set in the random multigraph $G^*(n,(d_i)_1^n)$.

Let $\tau_\infty$ be the unique value in $(0,\infty]$ such that
\begin{equation} \label{eq.tauinfty1}
\lambda \int_0^{\tau_\infty} \frac{e^{-2\sigma}}{\sum_k k p_k e^{-k\sigma}} \, d\sigma = 1.
\end{equation}
Then
\begin{equation} \label{eq.jamming1}
\frac{S^{(n)}_\infty}{n} \to \lambda \int_0^{\tau_\infty} e^{-2\sigma}
\frac{\sum_k p_k e^{-k\sigma}}{\sum_k k p_k e^{-k\sigma}} \, d\sigma \quad \mbox{in probability}.
\end{equation}

The same holds if $S^{(n)}_\infty$ is the size of a random greedy
independent set in the random graph $G (n,(d_i)_1^n)$, if we assume
additionally that
$\sumk k^2 n_k = O(n)$ as $n \to \infty$.
\end{theorem}

Since $S_\infty^{(n)} / n$ is bounded (by 1), it follows that the expectation
$\E S_\infty^{(n)} / n$ also tends to the limit in \eqref{eq.jamming1} under
the hypotheses in the theorem.

Our proof yields also the asymptotic degree distribution in the random
greedy independent set.

\begin{theorem} \label{thm.degrees}
Under the assumptions of Theorem \ref{thm.main},
let $S^{(n)}_\infty(k)$ denote the number of vertices of degree $k$ in the
random greedy independent set in the random multigraph $G^*(n,(d_i)_1^n)$.
Then, for each $k=0,1,\dots$,
\begin{equation} \label{eq.jammingk}
\frac{S^{(n)}_\infty(k)}{n} \to \lambda \int_0^{\tau_\infty} e^{-2\sigma}
\frac{p_k e^{-k\sigma}}{\sum_j j p_j e^{-j\sigma}} \, d\sigma \quad \mbox{in
  probability}.
\end{equation}

The same holds
in the random graph $G (n,(d_i)_1^n)$, if we assume
additionally that
$\sumk k^2 n_k = O(n)$ as $n \to \infty$.
\end{theorem}

\begin{remark}
We do not know whether the theorems hold also for the simple random graph
$G(n, (d_i)_1^n)$ without the additional hypothesis that $\sum_k k^2
n_k=O(n)$.
We leave this as an open problem.
(See Example \ref{Estar} for a counter example if we do not even assume that
$kn_k/n$ is uniformly summable.)
\end{remark}

One natural special case concerns a random $d$-regular graph, which is covered by the case where $p_d =1$, for some $d$.  It is not hard to see that the
jamming constant of a random 2-regular graph, in the limit as $n \to \infty$, is the same as that of a single cycle (or path), again in the limit as the
number of vertices tends to infinity.  An equivalent version of the greedy
process in this case is for ``cars'' to arrive sequentially, choose some
pair of
adjacent vertices on the cycle, and occupy both if they are both currently
empty.
This is then a discrete variant of the R\'enyi parking problem~\cite{Renyi}.
The limiting density of occupied vertices was first calculated by Flory~\cite{Flory} in 1939 to be $\frac12 (1-e^{-2})$.  See also Page~\cite{Page},
Evans~\cite{Evans}
and Gerin \cite{Gerin}.

The case of random $d$-regular graphs for $d \ge 3$
was first treated by Wormald~\cite{Wormald}, in a 1995 paper where he first gives a general scheme for his differential equations method, and independently by
Frieze and Suen~\cite{FS} for $d=3$.  Wormald gives an analysis of exactly the random process we consider, in the special case of $d$-regular graphs, and
obtains the result that the independent set constructed is,
w.h.p.\ (i.e., with probability tending to 1 as $n \to \infty$), of size
\begin{equation} \label{eq.d-reg}
\frac n 2 \left( 1 - \left( \frac{1}{d-1} \right)^{2/(d-2)} \right) + o(n).
\end{equation}
We shall verify shortly that our results give the formulae above in these cases.  Lauer and Wormald~\cite{LW}, and Gamarnik and Goldberg~\cite{GG}, extended
this result to give upper and lower bounds on the size of a random greedy independent set in any $d$-regular graph of girth at least~$g$, which
converge to the expression in~(\ref{eq.d-reg}) as~$g\to \infty$.

The problem of finding the size of a greedy independent set in a
more general
random graph with given degrees was first studied
in a recent preprint by Bermolen, Jonckheere and Moyal~\cite{BJM}.  Their
approach is superficially similar to ours, in that they construct the graph
and the independent set simultaneously, but their analysis is significantly
more complicated. (We explain the difference in more detail once we have
described our process in Section~\ref{sec:process}.)
Under the assumption that the 6th moment of the degree sequence is bounded,
they prove that their process is approximated by the unique solution of an
infinite-dimensional differential equation. The paper gives no explicit
form for the solution (except in the case of a
random $2$-regular graph, and for the Poisson distribution; in the latter
case, the authors substitute a Poisson distribution for the number of empty vertices of degree $k$ and show that this satisfies their equations), and the differential equation itself involves
the second moment of the degree sequence. The authors evaluate the solution
numerically in several explicit instances, and extract the jamming
constant.

We provide a simple system of approximating differential equations involving only the first moment of the degree sequence, and a time change where they have a simple and explicit solution.
We also require only first moment conditions on the degree sequence for our
result on the multigraph $G^*(n, (d_i)_1^n)$, and we assume that the 2nd
moment is uniformly bounded to transfer this result to the random graph
$G(n, (d_i)_1^n)$.

The simple greedy process for creating a random independent set is of interest for its own sake.  In chemistry and physics, this process is called
{\em random sequential adsorption}, and is studied, for instance, as a model of the deposition of a thin film of liquid onto a crystal.  See the surveys of
Evans~\cite{Evans} and Cadilhe, Ara\'ujo and Privman~\cite{CAP}, and the many further references therein, for information on applications of the
greedy algorithm and its variants.  In statistics, the greedy process is
known as {\em simple sequential inhibition}; see for instance
Diggle~\cite[\S 6.7]{Diggle}.
In these and other application areas, there is a ``continuum'' version, typically involving a greedy process for packing $d$-dimensional unit cubes
into $[0,M]^d$ for $M$ large.  In this process, unit cubes arrive in sequence, choose a location for their bottom corner uniformly at random in $[0,M-1]^d$,
and occupy the space if no already-placed cube overlaps it.
(The one-dimensional version is R\'enyi's car-parking process~\cite{Renyi}).  See Penrose~\cite{Penrose} for rigorous results on this car-parking
process in higher dimensions.  There is often also a discrete version, typically taking place on a regular lattice, where an object arrives and selects a
location on the lattice uniformly at random, and then inhibits later objects
from occupying neighbouring points.  Bermolen, Jonckheere and
Moyal~\cite{BJM} and Finch~\cite[\S 5.3]{Finch} also list a number of other
application areas, e.g., to linguistics, sociology and computer science,
again with further references.

Wormald~\cite{Wormald} also analysed a more complicated algorithm that produces, w.h.p., a larger independent set in a random $d$-regular
graph on $n$ vertices.  Subsequent interest in this line of research has indeed mostly been
focussed on finding good bounds on the  size  $I(G_{n,d})$
of a {\em largest} independent set in a random $d$-regular graph $G_{n,d}$,
and in particular its expectation  $\E I(G_{n,d})$:
a standard argument shows that $I(G_{n,d})$ is concentrated around its mean
with fluctuations of the order at most $\sqrt n$.
Bollob\'as~\cite{Bollobas} had earlier used a first moment argument to give
an upper bound on $\E I(G_{n,d})$ of the form $n \gamma^+(d) + o(n)$, where
$\gamma^+(d)$ is a function such that
\begin{equation} \label{eq.gamma}
\gamma^+(d) = \frac{2}{d} \left( \log d - \log \log d + 1 - \log 2 + o(1) \right)
\end{equation}
as $d \to \infty$ (whereas the greedy independent set has size approximately $n \log d / d$ for large $d$).  Frieze and \L uczak~\cite{FL}
gave a lower bound on $\E I(G_{n,d})$  of the form $n \gamma^-(d) + o(n)$, where $\gamma^-$ has the same asymptotic behaviour as in~(\ref{eq.gamma}).
Bayati, Gamarnik and Tetali~\cite{BGT} resolved the
long-standing open question of showing that there is, for each $d$, a constant $\gamma(d)$ such that $\E I(G_{n,d}) = n \gamma(d) + o(n)$.
Recently, Ding, Sly and Sun~\cite{DSS} showed that, for $d$ greater than some (large) $d_0$, $I(G_{n,d})$ has constant fluctuations about a function
of the form $\alpha_d n + \beta_d \log n$, where $\alpha_d$ and $\beta_d$ are constants whose values solve some explicitly given equations.
We are not aware of any work on the largest independent set in a more
general random graph with given vertex degrees.

\medskip

We now investigate the formulae of Theorem~\ref{thm.main}.

If $p_0+p_1 = 1$,
then we find that $\tau_\infty = \infty$, and the formula (\ref{eq.jamming1}) for the limit of $S_\infty/n$ evaluates to
$p_0+\frac12p_1$, which is the expected answer: indeed if a multigraph
has
$n_0$ isolated vertices,
$n_1=n-n_0-o(n)$ vertices of degree~1,
and $\frac12 n_1 +o(n)$ edges in total, then any maximal independent set has size $n_0+\frac12 n_1 + o(n)$.

We shall thus assume from now on that there is some $\ell \ge 2$ such that $p_\ell > 0$.  In this case, we see that
$$
\frac{\lambda e^{-2\sigma}}{\sum_k k p_k e^{-k\sigma}} = \frac{\sum_k k p_k e^{-2\sigma}}{\sum_k k p_k e^{-k\sigma}} > e^{-\sigma}
$$
for all $\sigma > 0$, and hence
$$
\int_0^\infty \frac{\lambda e^{-2\sigma}}{\sum_k k p_k e^{-k\sigma}} \, d\sigma > \int_0^\infty e^{-\sigma} \, d\sigma = 1.
$$
As the integrand is positive and bounded on finite intervals,
this implies that there is a unique finite value $\tau_\infty$ satisfying~(\ref{eq.tauinfty1}).

\medskip

We illustrate Theorem~\ref{thm.main} in several specific cases.  First, consider a random $d$-regular (multi)graph, where $p_d=1$.
Evaluating the integral in~(\ref{eq.tauinfty1}) and setting it equal to~1 gives
\begin{equation}
1 = \int_{\sigma = 0}^{\tau_\infty} e^{(d-2)\sigma} \, d\sigma
= \frac{1}{d-2} ( e^{(d-2)\tau_\infty} - 1),
\end{equation}
for $d \ge 3$, and so $\tau_\infty = \frac{\log(d-1)}{d-2}$.  For $d=2$ we obtain $\tau_\infty = 1$.
Now the formula in (\ref{eq.jamming1}) becomes
\begin{equation} \label{eq.jamming3}
\int_0^{\tau_\infty} e^{-2\sigma} \, d\sigma = \frac12 ( 1- e^{-2\tau_\infty} ) = \frac{1}{2} \left( 1 - \frac{1}{(d-1)^{2/(d-2)}} \right),
\end{equation}
for $d \ge 3$, and $\frac12 (1-e^{-2})$ for $d=2$.  This indeed agrees with Wormald's formula~(\ref{eq.d-reg}), as well as Flory's
formula~\cite{Flory} for $d =2$.

\medskip

We can also verify that Theorem~\ref{thm.main} gives the known answer in the
case of the random graph $G_{n,p}$, where $p = c/n$.  In this case,
the vertex degrees are random, but by conditioning on the vertex degrees
we can apply the results above, with the asymptotic Poisson degree distribution
$p_k = \frac{c^k e^{-c}}{k!}$, for each $k \ge 0$.  We may then
calculate that:
$$
\lambda = \sum_k k p_k = c; \quad \sum_k p_k e^{-k \sigma} = e^{-c} e^{ce^{-\sigma}}; \qquad \sum_k k p_k e^{-k \sigma} = c e^{-c} e^{-\sigma} e^{ce^{-\sigma}}.
$$
From~(\ref{eq.tauinfty1}), we find that
$$
1 = e^c \int_0^{\tau_\infty} e^{-\sigma} e^{-ce^{-\sigma}} \, d\sigma = \frac{e^c}{c} \left[ e^{-ce^{-{\tau_\infty}}} - e^{-c} \right],
$$
and rearranging gives
\begin{equation}
e^{-{\tau_\infty}} = 1 - \frac{\log (c+1)}{c}.
\end{equation}
We now have from (\ref{eq.jamming1}) that
\begin{equation}\label{gnp}
\frac{S_\infty}{n} \to \int_0^{\tau_\infty} e^{-\sigma} \, d\sigma
= 1-e^{-{\tau_\infty}} = \frac{\log (c+1)}{c},
\end{equation}
which agrees with the known value, which can be found from first principles
by a short calculation; see McDiarmid~\cite{McD}.

Furthermore, still considering $G_{n,p}$ with $p=c/n$, Theorem \ref{thm.degrees}
yields
\begin{equation}\label{gnp-k}
\frac{S_\infty(k)}{n}
\to \int_0^{\tau_\infty}\frac{c^k}{k!} e^{-(k+1)\sigma}e^{-ce^{-\sigma}}
\, d\sigma
= \frac{1}c \int_{c-\log(c+1)}^c \frac{x^k}{k!} e^{-x}\,dx.
\end{equation}
Combining \eqref{gnp} and \eqref{gnp-k}, we see that the asymptotic degree
distribution in the random greedy independent set can be described as a
mixture of $\Po(\mu)$, with parameter $\mu$ uniformly distributed in
$[c-\log(c+1),c]$.

\medskip

The structure of the remainder of the paper is as follows.  In Section~\ref{sec:process}, we give a full description of our process, which generates a random multigraph $G^*(n,(d_i)_1^n)$
with the given degrees simultaneously with a greedy independent set.  We
also specify which variables we track through the process. In
Section~\ref{sec:drift}, we calculate the drift of our process. In
Section~\ref{sec:analysis}, we write down the limit differential equations,
prove that the process converges to these equations, and prove
Theorem~\ref{thm.main}.
In \refS{sec:noui}, we show that, for the multigraph version, the
conclusions of Theorem~\ref{thm.main} still hold if $\lambda_n :=\sum_k k n_k/n$ tends to a constant $\lambda$ that may be larger
than $\sum_k k p_k$; we also consider the case where $\lambda_n$ tends to infinity.

\section{Description of process} \label{sec:process}

The key to our analysis is a process that generates the random multigraph and the random independent set in parallel.  The process analysed by
Bermolen, Jonckheere and Moyal~\cite{BJM} is also of this nature, but our process differs in one important way, as we point out below, and our choice
gives us significantly more tractable equations to solve.

Recall from the introduction that, in the configuration model, the random multigraph $G^*(n,(d_i)_1^n)$ is constructed as follows: for each $k$, and each of the $n_k$ vertices $i$ with
degree $d_i = k$, we associate $k$ {\em half-edges} with vertex $i$. These
half-edges are then combined into edges by a uniformly random
matching. Conditioned on $G^*(n,(d_i)_1^n)$ being simple (i.e.\ having no
loops or multiple edges), we obtain the random graph $G(n,(d_i)_1^n)$,
uniformly distributed over all graphs on $n$ vertices with given degree
sequence $(d_i)_1^n$.
Our process will generate $G^*(n,(d_i)_1^n)$ sequentially, by revealing the pairings of the half-edges at a chosen vertex as needed.
The freedom to pair the half-edges in any order allows one to study
other aspects of the random graph as we generate it, in this case a greedy independent set.  This theme has already been exploited many times, e.g., see~\cite{Wormald, JL1, JL2, JLW}.

\medskip

We analyse the following continuous-time Markovian process, which generates
a random multigraph on a fixed set $\cV = \{1, \dots, n\}$ of $n$ vertices
with pre-specified degrees,
so that vertex~$i$ has degree~$d_i$ for $i=1, \dots, n$,
along with an independent set $\cS$ in the multigraph.
At each time $t \ge 0$, the vertex set $\cV$ is partitioned into three classes:
\begin{enumerate}
\item [(a)] a set $\cS_t$ of vertices that have already been placed into the independent set, with all half-edges out of $\cS_t$ paired,
\item [(b)] a set $\cB_t$ of {\em blocked} vertices, where at least one half-edge has been paired with a half-edge from $\cS_t$,
\item [(c)] a set $\cE_t$ of {\em empty} vertices, from which no half-edge has yet been paired.
\end{enumerate}
At all times, the only paired edges are those with at least one endpoint in~$\cS_t$.
For $j =1, 2, \dots$, we set $\cE_t(j)$ to be the set of vertices in $\cE_t$ of degree~$j$.

Initially all vertices are empty, i.e., $\cE_0=\cV$.
Each vertex $v$
has an independent exponential clock, with rate~1.  When the clock of vertex $v \in \cE_t$
goes off, the vertex
is placed into the independent set and all its half-edges are
paired.  This results in the following changes:
\begin{enumerate}
\item [(a)] $v$ is moved from $\cE_t$ to $\cS_t$,
\item [(b)] each half-edge incident to $v$ is paired in turn with some other uniformly randomly chosen currently unpaired half-edge,
\item [(c)] all the vertices in $\cE_t$ where some half-edge has been paired
  with a half-edge from $v$ are moved to $\cB_t$.
\end{enumerate}
Note that some half-edges from $v$ may be paired with half-edges from $\cB_t$, or indeed with other half-edges from $v$: no change in the status
of a vertex results from such pairings.

The clocks of vertices in $\cB_t$ are ignored.

The process terminates when $\cE_t$ is empty.  At this point, there may still be some unpaired half-edges attached to blocked vertices: these may be paired off uniformly at random to complete the creation of the random multigraph.

The pairing generated is a uniform random pairing of all the half-edges.  The independent set generated in the random multigraph can also be
described as follows: vertices have clocks that go off in a random order,
and when the clock at any vertex goes off, it is placed in the
independent set if possible -- thus our process does generate a random greedy independent set in the random multigraph.

One particular feature of our process is that, when a vertex is moved from $\cE_t$ to $\cB_t$, we do not pair its half-edges.  This is in contrast
with the process studied by Bermolen, Jonckheere and Moyal~\cite{BJM}, where they reveal the neighbours of blocked vertices, meaning that the degress of empty vertices can change over time.

The variables we track in our analysis of the process are: $E_t(j) = |\cE_t(j)|$, the number of empty vertices of degree $j$
at time~$t$,  for each $j\ge0$, 
the total number $U_t$ of unpaired half-edges,
and the number $S_t = |\cS_t|$ of vertices that have so far been placed in the independent set.  We claim that the
vector $(U_t,E_t(0),E_t(1), \dots, S_t)$ is Markovian.  At each time~$t$, there are $E_t(j)$ clocks associated with empty vertices of degree~$j$; when the clock at one such vertex $v$ goes off,
its $j$ half-edges are paired uniformly at random within the pool of $U_t$ available half-edges,
so $U_t$ goes down by exactly $2j -2\ell$, where $\ell$ is the number of loops generated at $v$, which has a distribution that can be derived from
$|U_t|$ and the degree of $v$.  The distribution of the numbers of vertices of each $\cE_t(k)$ that are paired with one of the half-edges out of $v$
is a straightforward function of the vector given.  Meanwhile $S_t$ increases by one each time a clock at an empty vertex goes off.

\section{Drift}

\label{sec:drift}

We say that a stochastic process $X_t$, $t\ge0$, has \emph{drift} $Y_t$
if
$X_t-X_0-\int_0^t Y_s\,ds$
is a martingale.
(Or, more generally, a local martingale. In our cases, the processes are
for each $n$ bounded on finite intervals, so any local martingale is a
martingale.)

\begin{lemma}\label{LSdrift}
$S_t$ has drift
\begin{equation}
	\sumko  E_t(k).
  \end{equation}
\end{lemma}
\begin{proof}
  This is immediate, since $S_t$ increases by 1 each time the clock at an empty vertex
goes off, and they all go off with rate 1.
\end{proof}

\begin{lemma}\label{LUdrift}
$  U_t$ has drift
  \begin{equation}
	-\sumk k E_t(k) \Bigpar{2-\frac{k-1}{U_t-1}}.
  \end{equation}
\end{lemma}

\begin{proof}
  When the clock at a vertex of degree $k$ goes off, then the number of free half-edges decreases by
  $k+(k-2L)=2k-2L$, where $L$ is the number of loops created at the vertex.
  We have, conditionally on $U_t$,
\begin{equation}
  \E L = \binom k2 \frac{1}{U_t-1}
\end{equation}
and thus
\begin{equation}
  \begin{split}
  \E\#\set{\text{removed half-edges}}
&=2k-2\E L
=2k-\frac{k(k-1)}{U_t-1}	
\\&
=k\Bigpar{2-\frac{k-1}{U_t-1}}.
  \end{split}
\end{equation}
Now multiply by $E_t(k)$ and sum over $k$.
\end{proof}

Consider two distinct vertices $v$ and $w$,
of degrees $j$ and $k$, respectively,
in a configuration model with $u$ half-edges.
The probability that $v$ and $w$ are connected by at least one edge depends
only on $j$, $k$ and $u$; denote it by $p_{jk}(u)$.

\begin{lemma}\label{LEdrift}
  $E_t(k)$ has drift
  \begin{equation}
	-E_t(k)- \sumj p_{jk}(U_t) E_t(j)\bigpar{E_t(k)-\gd_{jk}}.
  \end{equation}
\end{lemma}

\begin{proof}
When the clock at a vertex in $\cE_t(j)$ (i.e., an empty
vertex of degree $j$) goes off, that vertex is removed from the set
$\cE_t(j)$,
which reduces $E_t(j)$ by 1.
Furthermore, when the clock at a vertex in $\cE_t(j)$ goes off, each empty vertex of degree
$k$ is joined to it with probability $p_{jk}(U_t)$ (since we may ignore the
half-edges already paired). Hence the expected total decrease of $E_t(k)$ when a
vertex in $\cE_t(j)$ goes off is
$p_{jk}(U_t)E_t(k)$ when $j\neq k$ and
$1+p_{jk}(U_t)(E_t(k)-1)$ when $j= k$. Now multiply by $E_t(j)$ and sum over
$j$.
\end{proof}

Let $(k)_m$ denote the falling factorials.

\begin{lemma}\label{Lpjk}
We have, with the sums really finite and extending only to $j\land k$,
\begin{equation}\label{pjk}
	\begin{split}
	p_{jk}(u)
&
= \summ\frac{(-1)^{m-1}}{m!} \frac{(j)_m(k)_m}{(u-1)(u-3)\dotsm (u-2m+1)}
\\&
= \summ\frac{(-1)^{m-1}}{m!} \frac{(j)_m(k)_m}{2^m((u-1)/2)_m}	  .
	\end{split}
  \end{equation}
Furthermore,
\begin{equation}
  \label{bonf}
\frac{jk}{u-1}
\ge p_{jk}(u)
\ge \frac{jk}{u-1} -\frac{j(j-1)k(k-1)}{2(u-1)(u-3)}.
\end{equation}
\end{lemma}

\begin{proof}
  The formula \eqref{pjk} is obtained by the inclusion-exclusion principle applied to
  the variable $X$ defined as the number of edges between the two given
  vertices $v$ and $w$; or equivalently, applied to the family of events
  that a given pair of half-edges at $v$ and $w$ are joined.
We have, for any $m\ge0$,
  \begin{equation}
	\E (X)_m = (j)_m(k)_m\frac{1}{(u-1)(u-3)\dotsm(u-2m+1)},
  \end{equation}
and \eqref{pjk} is given by the standard formula
(see \citet[IV.1]{FellerI}, in an equivalent form)
\begin{equation}
\Pr(X\ge 1)
=\summ (-1)^{m-1}\E \binom Xm
=\summ \frac{(-1)^{m-1}}{m!}\E (X)_m  .
\end{equation}

Similarly, the inequalities \eqref{bonf} are instances of the Bonferroni
inequalities (see \citet[IV.5(c)]{FellerI}).
\end{proof}

\begin{remark}
  The sum in \eqref{pjk} is (a minor modification of) a hypergeometric sum;
in terms of the hypergeometric function ${}_2F_1$, we have
\begin{equation}
  p_{jk}(u) = 1- {}_2F_1\bigpar{-j,-k;\tfrac{1-u}2;\tfrac12}.
\end{equation}
\end{remark}

\section{Analysis: proof of Theorem~\ref{thm.main}}

\label{sec:analysis}

First, recall that $\sum_k k n_k/n \to \lambda$. Since it suffices to
consider large $n$, we may without loss of generality assume that
\begin{equation}
  \label{2lambda}
\sum_k kn_k \le 2\lambda n.
\end{equation}

Next,
by Lemmas \ref{LUdrift} and \ref{LEdrift},
we can write
\begin{equation}\label{mt}
U_t = U_0 -\int_0^t \sumk k E_s(k) \Bigpar{2-\frac{k-1}{U_s-1}}\,ds + M_t,
\end{equation}
and, for each $k \in \Z^+$,
\begin{equation}\label{mtk}
E_t(k) = E_0(k)-\int_0^t E_s(k)\,ds- \int_0^t \sumj p_{jk}(U_s) E_s(j)\bigpar{E_s(k)-\gd_{jk}}\,ds +M_t(k),
\end{equation}
where $M_t$ and each $M_t(k)$ is a martingale.

Dividing by $n$,
\begin{equation}
\label{eq-u}
\frac{U_t}{n} = \frac{U_0}{n}-\int_0^t \sumk \frac{k E_s(k)}{n} \Bigpar{2-\frac{k-1}{U_s-1}}\,ds + \frac{M_t}{n},
\end{equation}
and, for each $k$,
\begin{equation}
\label{eq-ek}
\frac{E_t(k)}{n} = \frac{E_0(k)}{n}-\int_0^t \frac{E_s(k)}{n}\,ds- \int_0^t \sumj p_{jk}(U_s) \frac{E_s(j)}{n}\bigpar{E_s(k)-\gd_{jk}}\,ds +\frac{M_t(k)}{n}.
\end{equation}
The martingale $M_t$ has by \eqref{mt} locally finite variation, and
hence its quadratic variation
$[M]_t$ is given by
\begin{equation}
  \label{qM}
[M]_t
= \sum_{0\le s\le t} (\Delta M_s)^2
= \sum_{0\le s\le t} (\Delta U_s)^2
\le \sum_{s \ge 0} (\Delta U_s)^2 \le \sum_j (2j)^2 n_j = o(n^2),
\end{equation}
since $\sumk k n_k/n$ is uniformly summable.
(In fact, $[M]_t = O(n)$ if the second moment of the degree distribution is
uniformly bounded.)
Likewise, for each $k$,
\begin{equation}\label{qMk}
[M(k)]_t \le \sum_{s \ge 0} (\Delta E_s(k))^2\le \sum_j (j+1)^2 n_j = o(n^2),
\end{equation}
with the bound being $O(n)$ if the second moment of the degree distribution
is uniformly bounded.
Doob's inequality then gives
$\sup_{t \ge 0} \abs{M_{t}} = \op(n)$ and, for each $k$,
$\sup_{t \ge 0}\abs{M_{t}(k)} = \op(n)$.

Now the integrand in~(\ref{eq-u}) is bounded above by
$2\sumk kn_k/n$, which is less than $4 \lambda$ by \eqref{2lambda}.
Also, the integrand in the first integral in~(\ref{eq-ek}) is
bounded above by 1, and the integrand in the second integral is bounded
above by $4 \lambda$, since
$p_{jk} (U_t) \le jk/(U_t-1)\le 2jk/U_t$ and $\sumj jE_s(j)\le U_s$
(there are at least as many
free half-edges, as there are free half-edges at empty vertices),

It follows that $(U_t- M_t)/n$, $n \ge 1$, is a uniformly Lipschitz family, and
it is also uniformly bounded on each finite interval $[0,t_0]$. Likewise,
for each $k$, $(E_t(k) - M_t(k))/n$, $n \ge 1$, is a uniformly Lipschitz
family, and
it is also uniformly bounded on each finite interval $[0,t_0]$.
Thus, the Arzela--Ascoli theorem implies that each of the above families is tight in $C[0,t_0]$ for
any $t_0> 0$ \cite[Theorems A2.1 and  16.5]{Kallenberg},
and so also in $C[0,\infty)$ \cite[Theorem 16.6]{Kallenberg}.
This then implies that, for each of the above processes, there exists a
subsequence along which the process converges in distribution in
$C[0,\infty)$.
Since there are countably many processes, we can find a common subsequence
where
\begin{equation}\label{uemconv}
 \frac{U_t - M_t}n\to u_t \qquad \text{and each}\qquad
\frac{E_t(k)-M_t(k)}n\to e_t(k)
\end{equation}
in distribution
in $C[0,\infty)$, for some random continuous functions $u_t$ and $e_t(k)$.
By the Skorokhod coupling lemma \cite[Theorem 4.30]{Kallenberg}
we may assume that the limits in \eqref{uemconv} hold almost surely
(in $C[0,\infty)$, i.e., uniformly on compact sets),
and also that
$\sup_{t \ge 0} \abs{M_{t}}/n\asto 0$
and
$\sup_{t \ge 0} \abs{M_{t}(k)}/n\asto 0$.
Hence, along the subsequence, a.s.
\begin{equation}\label{ueconv}
 \frac{U_t }n\to u_t \qquad \text{and each}\qquad
\frac{E_t(k)}n\to e_t(k)
\end{equation}
uniformly on compact sets,
with continuous limits $u_t$ and $e_t(k)$,  $k =0,1,2, \ldots$.
Since $U_t \ge
0$ and $E_t(k) \ge 0$,  we must have $u_t \ge 0$ and $e_t(k)
\ge 0$ for all $t$ and $k$.
Clearly, $u_0 = \lambda$ and $e_0(k) = p_k$.

Next note that
we have, by \eqref{mt}, for $0\le r<t<\infty$,
as $\sum_k k E_s (k) \le U_s$,
\begin{equation}
  \begin{split}
\frac{U_t}{n} -\frac{U_r}n
& \ge -2\int_r^t \sumk \frac{k E_s(k)}{n}\,ds + \frac{M_t}{n}-\frac{M_r}{n}
\\&
\ge - 2 \int_r^t \frac{U_s}n \,ds + \frac{M_t}{n}-\frac{M_r}{n}.	
  \end{split}
\end{equation}
This, and the fact that $U_t$ is non-increasing, then implies that
\begin{equation}\label{Deltau}
0\ge
  u_t-u_r \ge  - 2 \int_r^t u_s \,ds.
\end{equation}
This (or
\eqref{uemconv} and the argument preceding it)
shows that $u_t$ is a Lipschitz function.
Thus $u_t$ is differentiable a.e., and \eqref{Deltau} implies
$\frac{d}{dt}u_t\ge-2u_t$  and thus $\frac{d}{dt}(e^{2t} u_t)\ge0$ a.e.
A Lipschitz function is absolutely continuous and thus the integral of its
derivative; hence it follows that $e^{2t}u_t$ is non-decreasing
and hence that $u_t \ge e^{-2t} u_0 = \lambda e^{-2t} > 0$ for all $t$.

In particular, \eqref{ueconv} implies that  along the subsequence, a.s.\
$U_t\to\infty$ for every fixed $t$.

Thus,
a.s., along the subsequence,
$(k-1)/(U_s-1) \to 0$
for every fixed $k$ and $s <\infty$,
and so, using \eqref{ueconv},
\begin{equation}
  \label{marthie}
a_s(k):= \frac{k E_s (k)}{n} \Big ( 2 - \frac{k-1}{U_s-1}\Big ) \to 2 k e_s (k).
\end{equation}
Furthermore,
\begin{equation}\label{malby}
a_s (k) \le \frac{2 k E_0(k)}{n} = \frac{2kn_k}{n}.
\end{equation}
By assumption, $kn_k/n$ is uniformly summable, and by \eqref{malby}
so is  $a_s(k)$. Hence \eqref{marthie} yields
\begin{equation}\label{elf1}
  A_s := \sumk a_s (k) \to 2 \sumk k e_s (k).
\end{equation}
Furthermore, by \eqref{malby} and \eqref{2lambda},
for all $s\ge0$,
\begin{equation}
A_s \le \frac{2 \sumk k n_k}{n} \le 4\lambda.
\end{equation}
Hence,
by dominated convergence, for any $t<\infty$, a.s.
\begin{equation}
\int_0^t A_s \,ds \to 2 \int_0^t \sumk k e_s(k) \,ds.
\end{equation}
Note that the integral in \eqref{eq-u} is $\int_0^t A_s\,ds$.
Hence, taking the limit in \eqref{eq-u} as \ntoo{} along the subsequence
yields, using \eqref{ueconv} and $M_t/n\to0$ a.s.,
\begin{equation}
\label{eq.int.u}
u_t = \lambda - 2\int_0^t \sumk k e_s(k) \, ds.
\end{equation}

Now, fix $k \in \Z^+$, and, for each $j \in \Z^+$, let
\begin{equation}
b_s (j) = p_{jk} (U_s) \frac{E_s (j)}{n} (E_s (k) - \delta_{jk}).
\end{equation}
By Lemma \ref{Lpjk} and \eqref{ueconv}, a.s.,
for each fixed $s\ge0$ and $j\in\Z^+$,
\begin{equation}
n p_{jk}(U_s)
= \frac{njk}{U_s-1} \Bigpar{1 + O \Bigpar{\frac{jk}{U_s-1} }}
\to \frac{jk}{u_s}
\end{equation}
and
\begin{equation}
\frac{E_s(j)}{n} \frac{E_s(k) - \delta_{jk}}{n}
\to e_s (j) e_s (k).
\end{equation}
Hence, a.s., for each fixed $s\ge0$ and $j\in\Z^+$,
\begin{equation}
  \label{bjlim}
b_s(j)\to \frac{jk}{u_s}e_s(j)e_s(k).
\end{equation}
Furthermore,
if $U_s=0$ then $E_s(j)=0$ and $b_s(j)=0$, and if $U_s>0$ then
\begin{equation}\label{bjbound}
0\le b_s (j) \le \frac{jk}{U_s -1} \frac{E_s(j)E_s(k)}{n}
\le  2\frac{kE_s(k)}{U_s} \frac{j E_s(j)}{n}
\le 2\frac{j n_j}n.
\end{equation}
Thus $b_s (j)$ is uniformly summable and
\eqref{bjlim} yields
\begin{equation}\label{elf2}
\sum_j b_s (j) \to k e_s (k)\sum_j \frac{j e_s(j)}{u_s}.
\end{equation}
Moreover, by \eqref{bjbound} and \eqref{2lambda}, for every $s$,
\begin{equation}
\sum_j b_s (j)
\le \frac{2\sum_j j n_j}{n} \le 4 \lambda.
\end{equation}
Consequently, by dominated convergence, for any $t<\infty$,
\begin{equation}
\int_0^t \sum_j b_s (j) \,ds
\to \int_0^t k e_s(k) \sum_j \frac{j e_s(j)}{u_s} \,ds \quad \mbox{a.s.}
\end{equation}

The second integral in \eqref{eq-ek} is $\int_0^t\sum_j b_s(j)\,ds$, and for
the first integral we have directly by \eqref{ueconv} and dominated
convergence
\begin{equation}
  \int_0^t\frac{E_s(k)}n\,ds\to\int_0^t e_s(k)\,ds \quad \mbox{a.s.}
\end{equation}
Since also $M_t(k)/n\to0$ a.s., we thus see from \eqref{eq-ek} that
\begin{equation}
\label{eq.int.e}
e_t(k) = p_k-\int_0^t e_s(k)\,ds- \int_0^t k e_s(k)\frac{\sumj j e_s(j)}{u_t}\,ds.
\end{equation}

In other words, we have shown that the subsequential \as{} limit
$(u_t,e_t(k): k =0,1, \ldots)$ of $(U_t/n, E_t(k)/n: k=0,1,\ldots)$
\as{} must
satisfy the following system of equations:
\begin{align}
u_t & =  \lambda - 2\int_0^t \sumk k e_s(k)\,ds \label{eq.int-u}
\\
e_t(k) & =  p_k-\int_0^t e_s(k)\,ds
- \int_0^t k e_s(k)\frac{\sumj j e_s(j)}{u_s}\,ds,
\quad k \in \Z^+. \label{eq.int-e}
\end{align}
Below, we will show that the
equations~\eqref{eq.int-u}--\eqref{eq.int-e} have a unique
solution $(u_t,e_t(k): k =0,1, \ldots)$.
This means that
$u_t$ and $e_t(k)$, which a priori are random, in fact are deterministic.
Moreover,
since the limits $u_t$ and $e_t(k)$ are continuous,
the limits uniformly on compact sets in \eqref{ueconv}
are equivalent to convergence in the Skorokhod space $D[0,\infty)$.
Consequently, the argument above shows that each  subsequence
of $(U_t/n, E_t(k)/n: k=0,1,\ldots)$
has a subsequence which converges in distribution
in the Skorokhod topology, and that each convergent subsequence
converges to the same $(u_t,e_t(k): k =0,1,\ldots)$.
This implies that
the whole sequence $(U_t/n, E_t(k)/n: k=0,1,\ldots)$
must in fact converge to
$(u_t,e_t(k): k =0,1, \ldots)$
in distribution in the Skorokhod topology.
Since the limit is deterministic, the convergence holds in probability.

We proceed to solve  \eqref{eq.int-u}--\eqref{eq.int-e}.
(This is a system of equations of deterministic functions; there is no randomness
in this part of the proof.)
We use also the facts $e_t(k)\ge0$ and $u_t>0$ established above.

First of all, each $e_s(k)$ is continuous.
Furthermore, \eg{} by \eqref{eq.int-e},
$e_s(k)$ is bounded by $e_0(k)=p_k$ with
$\sum_k kp_k<\infty$.
Hence the sum $\sum_k k e_s(k)$ converges uniformly and thus this sum is
continuous. Since also $u_s$ is continuous and $u_s>0$, the integrands in
the integrals in \eqref{eq.int-u}--\eqref{eq.int-e} are continuous.
Consequently, the functions $u_t$ and $e_t(k)$ are continuously
differentiable and the integral equations
\eqref{eq.int-u}--\eqref{eq.int-e} can be written as a system of
differential equations
\begin{align}
\frac{du_t}{dt} & =  - 2 \sumk k e_t(k) \label{eq.diff-u}
\\
\frac{de_t(k)}{dt} & =  - e_t(k)
-  k e_t(k)\frac{\sumj j e_t(j)}{u_t}, \quad k \in \Z^+. \label{eq.diff-e}
\end{align}
with the initial conditions $u_0=\lambda$ and $e_0(k)=p_k$.

Note that the system \eqref{eq.diff-u}--\eqref{eq.diff-e}
is infinite, and it is not a priori obvious that it
has a solution, or that the
solution is unique. The system is not obviously Lipschitz with respect
to any of the usual norms on sequence spaces.
Fortunately, it is possible to decouple the
system via a change of variables and a time-change, leaving us with an
explicit solution in terms of the new variables.

We make the change of variables
\begin{equation}
  \label{he}
h_t(j) = e^t e_t(j),
\end{equation}
for each~$j$.  Note
that $\displaystyle \frac{dh_t(j)}{dt} = e^t \frac{de_t(j)}{dt} + e^t
e_t(j)$, so, from
\eqref{eq.diff-e}
we obtain
\begin{equation} \label{eq.htj}
\frac{dh_t(k)}{dt} = - k h_t(k) \frac{\sum_j j e_t(j)}{u_t}.
\end{equation}
Now we rescale time by introducing a new time variable $\tau=\tau_t$ such that
\begin{equation} \label{eq.time-change}
\frac{d\tau}{dt} = \frac{\sum_j j e_t(j)}{u_t},
\end{equation}
with $\tau_0 = 0$. This is well defined as $u_t > 0$ for all $t$.
Note that $\sum_jje_t(j)\le\sum_jjp_j=\lambda$; hence \eqref{eq.diff-e}
implies that for any $T>0$ and $t\in[0,T]$,
\begin{equation}\label{ebb}
\frac{de_t(k)}{dt}  \ge  - \Bigpar{1+\frac{\lambda}{u_T}k}e_t(k)
\end{equation}
and thus $e_t(k)\ge p_k \exp\bigset{-\bigpar{1+\xfrac{k\lambda}{u_T}}t}$ for
$t\in[0,T]$. In particular, if $p_k>0$, then $e_t(k)>0$ for all $t\ge0$.
Hence,
$\sumk k e_t (k) > 0$ for all $t$, and thus $\xfrac{d\tau}{dt} > 0$ for all
$t$, so $\tau$ is a strictly increasing function of~$t$.
Consequently, the mapping $t\mapsto\tau_t$ is a bijection of $[0,\infty)$
  onto $[0,\tau_\infty)$ for some
$\tau_\infty \in (0,\infty]$;
furthermore, this map and its inverse are both continuously differentiable.

We can thus regard the
functions $u_t$, $e_t(j)$ and $h_t(j)$ as functions of
$\tau\in[0,\tau_\infty)$; we denote
these simply by $u_\tau$, $e_\tau(j)$ and $h_\tau(j)$.
With this convention, we obtain from \eqref{eq.diff-u}, \eqref{eq.htj}
and \eqref{eq.time-change}
\begin{align} 
\frac{du_{\tau}}{d\tau}
& =  \frac{du}{dt} \frac{dt}{d \tau}
= - 2 u_\tau, \label{eq.diff-tau-u}\\
\frac{dh_{\tau}(j)}{d\tau}
& =\frac{dh(j)}{dt} \frac{dt}{d \tau}
=  - j h_\tau(j), \quad j \in \Z^+, \label{eq.diff-tau-h}
\end{align}
subject to initial conditions $u_0 = \lambda$ and $h_0(j) = p_j$.
The system~\eqref{eq.diff-tau-u}--\eqref{eq.diff-tau-h} has the obvious
unique solution
\begin{align}
u_\tau & =  \lambda e^{-2\tau} \label{eq.sol-u} \\
\quad h_\tau(j) & =  p_j e^{-j \tau}, \quad j \in \Z^+. \label{eq.sol-h}
\end{align}
Substituting \eqref{he} and \eqref{eq.sol-u}--\eqref{eq.sol-h}
into \eqref{eq.time-change}
yields
\begin{equation}
\frac{d\tau}{dt} = e^{-t} \frac{\sum_k k p_k e^{-k\tau}}{\lambda e^{-2\tau}},
\end{equation}
and separating the variables and using the boundary conditions gives
\begin{equation} \label{eq.time-change2}
1 - e^{-t} = \int_0^t e^{-s} \, ds
= \int_0^{\tau_t} \frac{\lambda e^{-2\sigma}}{\sum_k k p_k e^{-k\sigma}}
\, d\sigma.
\end{equation}
Since the integrand is positive, this determines $\tau_t$ uniquely for every
$t\in[0,\infty)$, and thus $u_t$, $h_t(j)$ and $e_t(j)$
are determined by  \eqref{eq.sol-u}--\eqref{eq.sol-h} and \eqref{he}.
This completes the proof that the equations
\eqref{eq.int-u}--\eqref{eq.int-e} have a unique solution, at least assuming
$e_t(k)\ge0$ and $u_t>0$, which any subsequential limit of our process must
satisfy, as shown above.

Furthermore, letting \ttoo{} in \eqref{eq.time-change2},
\begin{equation} \label{eq.tauinfty}
\int_0^{\tau_\infty} \frac{\lambda e^{-2\sigma}}{\sum_k k p_k e^{-k\sigma}} \, d\sigma = 1,
\end{equation}
so $\tau_\infty$ is the value given by \eqref{eq.tauinfty1}.
As we saw already in Section~\ref{sec:intro},  this
determines
$\tau_\infty$ uniquely.

It is time to consider the variable that we really are interested in,
viz.~$S_t$.
By Lemma \ref{LSdrift} (and $S_0=0$),
we have, in analogy with \eqref{mt}--\eqref{mtk},
\begin{equation}\label{mts}
S_t = \int_0^t \sumko E_s(k)\,ds + \tM_t,
\end{equation}
for some martingale $\tM_t$.
All jumps $\Delta S_t$ are equal to 1, and it follows as in \eqref{qM}
and \eqref{qMk} that
\begin{equation}
  [\tM]_t\le \sum_{s\ge0}(\Delta S_s)^2 \le n
\end{equation}
and Doob's inequality  gives
$\sup_{t \ge 0} \abs{\tM_{t}} = \op(n)$.
Furthermore,
we have shown that \eqref{ueconv} holds (for the full sequence) with
convergence in probability in $D[0,\infty)$, and again we may, by the
  Skorokhod coupling lemma, assume that the convergence holds \as{} in
  $D[0,\infty)$, and thus uniformly on compact intervals.
Since
$E_s(k)/n\le n_k/n$ and $\sumko n_k/n$ is uniformly summable (since
$\sum_k kn_k/n$ is), it then follows that
\begin{equation}\label{sumElim}
  \sumko\frac{E_s(k)}n \to \sumko e_s(k) \quad \mbox{a.s.}
\end{equation}
for every $s$. Furthermore, trivially $\sumko E_s(k)/n\le1$;
hence \eqref{mts} and \eqref{sumElim} yield, by dominated convergence, a.s.
\begin{equation}\label{lims}
\frac{S_t}n \to s_t:=\int_0^t \sumko e_s(k)\,ds
\end{equation}
for any $t<\infty$.
We want to extend this to $t=\infty$.

Since $S_t\le n$, we see from \eqref{lims} that $s_t\le1$ for every $t<\infty$;
hence
\begin{equation}
  \label{soo}
s_\infty = \lim_{t\to \infty} s_t = \int_0^\infty \sumko e_t(k) \, dt
\end{equation}
is finite. This can also be seen directly since $\sumko e_t(k)\le e^{-t}$
by \eqref{he} and \eqref{eq.sol-h}.
In particular, $\sumko e_t(k)\to0$ as \ttoo.

Let $\eps>0$. Then there exists $T=T(\eps)<\infty$ such that
$s_\infty-s_{T}<\eps$ and  $\sumko e_{T}(k)< \eps$.
Now,
\begin{equation}\label{tri}
\lrabs{\frac{S_{\infty}}{n} - s_{\infty}}
\le \lrabs{\frac{S_{\infty}-S_{T}}{n}}
+ \lrabs{\frac{S_{T}}{n} - s_{T}}
+ \bigabs{s_{T} - s_{\infty}}.
\end{equation}
The last term in \eqref{tri} is less than $\eps$ by our choice of $T$.
Furthermore, by \eqref{lims}, the term $|S_T/n-s_T|$ tends to 0 a.s., and
thus in probability.
Hence, $|S_T/n-s_T|\le \eps$
w.h.p.\ (meaning, as before, with probability tending to 1 as \ntoo).
Moreover, for any $t$,
$0\le S_\infty-S_t\le \sumko E_t(k)$, and
by \eqref{sumElim},
\begin{equation}
    \sumko\frac{E_{T}(k)}n \to \sumko e_T(k)<\eps
\end{equation}
a.s.\ and thus in probability.
Hence also
$(S_\infty-S_{T})/n<\eps$ w.h.p.
Consequently, \eqref{tri} shows that $|S_\infty/n-s_\infty|<3\eps$ w.h.p.
Since $\eps>0$ is arbitrary, we have shown that
\begin{equation}\label{soox}
  \frac{S_\infty}n\to s_\infty
\end{equation}
in probability, and it remains only to show that $s_\infty$ equals the
constant in \eqref{eq.jamming1}.

In order
to obtain the jamming constant $s_\infty$, we use \eqref{soo} and
make the same change of variables $t \mapsto \tau$ as before.
Thus, using \eqref{eq.time-change},
\begin{equation}\label{soo=}
  \begin{split}
s_\infty &
= \int_0^\infty \sum_k e_t(k) \, dt
= \int_0^{\tau_\infty} \sum_k e_\tau(k) \frac{u_\tau}{\sum_k k e_\tau(k)} \,
d\tau
\\&
= \int_0^{\tau_\infty} \frac{u_\tau \sum_k h_\tau(k)}{\sum_k k h_\tau(k)} \, d\tau.	
  \end{split}
\end{equation}
Substituting our explicit expressions
\eqref{eq.sol-u} and \eqref{eq.sol-h}
for $u_\tau$ and the $h_\tau(k)$
yields
\begin{equation} \label{eq.jamming}
s_\infty = \lambda \int_0^{\tau_\infty} e^{-2\tau} \frac{\sum_k p_k e^{-k\tau}}{\sum_k k p_k e^{-k\tau}} \, d\tau,
\end{equation}
as in Theorem~\ref{thm.main}.

This completes the proof of Theorem \ref{thm.main}.
Theorem \ref{thm.degrees} follows by simple modifications of the final part
of the proof above:
The drift of $S_t(k)$ is $E_t(k)$ and it follows as in
\eqref{lims} that
\begin{equation}\label{limsk}
\frac{S_t(k)}n \to s_t(k):=\int_0^t e_s(k)\,ds
\end{equation}
for any $t<\infty$. Arguing as in \eqref{tri}--\eqref{soox},
we see that \eqref{limsk}
holds for $t=\infty$ too, and \eqref{eq.jammingk} follows in analogy with
\eqref{soo=} and \eqref{eq.jamming}.

\section{Without uniform integrability}
\label{sec:noui}

In this section, we discuss what happens when we relax the condition of
uniform integrability of the degree distribution, i.e.,
uniform summability of $kn_k/n$.
This is relevant only for the multigraph case,
since for the graph case we have to assume $\sum_k k^2 n_k=O(n)$ for our proof,
and then $kn_k/n$ is always uniformly summable.
Recall our convention that looped vertices are eligible for the independent set: for part of the range of
heavy-tailed distributions we consider here, many vertices of larger degree will have loops, so this
choice has a significant impact on the behaviour,
see Example \ref{Estar}.

We still assume that the number $n_k$ of vertices of
degree $k$ satisfies $n_k/n \to p_k$ for $k = 0,1, \ldots$, where
$p_k\ge0$, and to avoid trivialities we assume $p_0<1$,
but we allow $\sumko p_k<1$.
We denote the average degree by $\gln:=\sumk k n_k/n$
and assume that $\gln$ converges to some $\lambda\le\infty$.

We first treat the case when $\lambda<\infty$.
Then the degree distributions are tight,
so their limit $(p_k)_0^{\infty}$ is a
probability distribution; furthermore, by Fatou's lemma,
the mean $\mu = \sumk k p_k$ of this limiting degree distribution
satisfies $\mu\le\lambda<\infty$, so
the probability distribution
$(p_k)_0^{\infty}$ has a finite and non-zero mean $\mu$.
(It is possible that $\mu < \lambda$.)

The proof in Section~\ref{sec:analysis} uses uniform summability of $kn_k/n$
to show
that the quadratic variation of martingales $(M_t)$ and $(M_t(k))$ is
$o(n^2)$ in \eqref{qM} and \eqref{qMk}; furthermore, we use it to justify
taking limits under the summation sign in \eqref{elf1}, \eqref{elf2} and
\eqref{sumElim}. We give in this section an alternative (but longer)
argument for the multigraph case that does not use the uniform summability.

We let $c_1, C_1,\dots$ denote positive constants that may depend on the
collection
$\set{n_k(n)}_{n,k\ge1}$ and $(p_k)_0^\infty$, but not on $n$.
We replace \eqref{2lambda} by $\sum_k kn_k \le \CC n$. \CCdef\CCmu

\begin{lemma}
  \label{LEE}
There exists $\cc>0$ such that
for every $t\in[0,1]$,
\begin{equation}\label{lee}
  \E E_t(k) \le 2 n_k e^{-\ccx kt},\ccdef\cclee
\qquad k\in \Z^+.
\end{equation}
\end{lemma}
\begin{proof}
We may assume that $n$ is large.
(For any fixed $n$, $n_k>0$
only for finitely many $k$; furthermore,
 $\E E_t(k)\le e^{-t}n_k$ and
\eqref{lee} follows if $\cclee$ is small enough.)

Assume $0<t\le1$.
Fix some $\ell\ge1$ with $p_\ell>0$, and consider only $n$ that are so large
that $\frac12p_\ell < n_\ell/n < 2p_\ell$.
Let $\gd:=p_{\ell}/9$, and suppose also that $\ceil{\gd n} \le np_\ell/8$.
Define the stopping time $\tau_t$ as the time that we create the
$\ceil{\gd t n}$-th edge (by pairing a half-edge with another).
Even if there are several edges created at the same time, we regard them as
created one by one, separated by infinitesimal time intervals, and we stop
when exactly the right number is created.
Until we stop, at most $2\ceil{\gd t n}\le 2\ceil{\gd n}$ vertices have
gone off or become blocked; in particular,  at each time before we stop there are at least
$n_\ell-2\ceil{\gd n}\ge\frac{1}4p_\ell n$ empty vertices of degree $\ell$.
Each time one of these
empty vertices goes off, at least one edge is created. (In fact, at least
$\ceil{(\ell+1)/2}$ edges.) Hence, until we stop,
the times between the creations of the edges can be dominated by \iid{}
exponential variables with
mean $1/(\cc n)$ (with $\ccx=p_\ell/4>\gd$), and thus the creation of edges
dominates a Poisson
process with intensity $\ccx n$; consequently,
by a Chernoff estimate,
\begin{equation}\label{ch}
  \Pr(\tau_t >t) \le \Pr\bigpar{\Po(\ccx nt)<\gd tn}
\le e^{-\cc tn}.\ccdef{\cctau}
\end{equation}

Now consider a vertex $v$ of degree $k$. As long as $v$ is empty, it has
probability $k/(U_t-1)\ge k/(\CCmu n)$ of being blocked at each pairing of a
half-edge.
Hence,
\begin{equation}
  \Pr(v \text{ is empty at $\tau_t$})
\le \Bigpar{1-\frac{k}{\CCmu n}}^{\gd t n} \le e^{-\cc kt},\ccdef{\ccempty}
\end{equation}
and thus
\begin{equation}\label{etau}
  \E E_{\tau_t}(k) \le e^{-\ccempty kt}n_k.
\end{equation}
Consequently, using \eqref{ch} and \eqref{etau},
\begin{equation}
  \E E_{t}(k) \le e^{-\ccempty kt}n_k+\Pr(\tau_t>t)n_k
\le \bigpar{e^{-\ccempty kt}+ e^{-\cctau nt}}n_k.
\end{equation}
Since $n_k>0$ only if $k\le \CCmu n$, the result follows.
\end{proof}

\begin{lemma}
  \label{LEEx}
Let $E^*(k)$ be the number of vertices of degree $k$ that are empty when
they go off. Then
\begin{equation}
  \E E^*(k)\le \CC\frac{n_k}k. \CCdef\CCleex
\end{equation}
\end{lemma}
\begin{proof}
  The expected number of empty vertices of degree $k$ that go off at a time
  in $[0,1]$ is, using  \refL{LEE},
  \begin{equation}
\int_0^1\E E_k(t)\,dt \le 2n_k\int_0^1 e^{-\cclee kt}\,dt \le \CC\frac{n_k}k.	
  \end{equation}
Furthermore, also by \refL{LEE}, the expected number of empty vertices of
degree $k$ that go off after 1 is at most
\begin{equation}
  \E E_1(k)\le 2n_k e^{-\cclee k} \le \CC\frac{n_k}k.
\end{equation}
\end{proof}

\begin{lemma}
  For the martingales $M$ and $M(k)$ in \eqref{qM} and \eqref{qMk},
  \begin{equation}
	\E [M]_\infty = O(n),\qquad \E [M(k)]_\infty = O(n).	
  \end{equation}
\end{lemma}

\begin{proof}
  Since $M$ jumps by at most $2j$ when an empty vertex of degree $j$ goes
  off, and $M$ otherwise is continuous, we have, \cf{} \eqref{qM},
$[M]_\infty \le \sum_j (2j)^2 E^*(j)$. Thus \refL{LEEx} implies
  \begin{equation}
	\E [M]_\infty \le \sumj (2j)^2 \E E^*(j) \le 4\CCleex\sumj j n_j
\le
\CC n.
  \end{equation}
Similarly, \cf{} \eqref{qMk}, $[M(k)]_\infty\le \sumjo (j+1)^2 E^*(j)$
and $	\E [M(k)]_\infty \le \CC n$ follows.
\end{proof}

This gives the estimates we need for the martingales.

Furthermore, \refL{LEE} implies that for any fixed integer $m\ge1$ and $t>0$,
$\E \sum_k k^m E_t(k)/n \le 2 \sum_k k^m e^{-\cclee k (1 \land t)}=O(1)$.
Hence the random variables $\sum_k k^m E_t(k)/n$ are tight, so, when taking
subsequences in the proof, we may also assume that
$\sum_k k^m E_t(k)/n$ converges in distribution, for any fixed $m$ (or all
$m$) and, say, any rational $t$.
Thus, when using the Skorokhod coupling lemma, we may also assume that
$\sum_k k^m E_t(k)/n$ converges a.s. In particular, choosing $m=2$,
$\sum_k k^2 E_t(k)/n$ is \as{} bounded for each fixed rational $t$, and
thus for each fixed $t>0$, since the sum is a decreasing function of $t$. This implies
uniform summability of $k E_s(k)/n$, for any fixed $s>0$, and thus uniform
summability in \eqref{elf1} and \eqref{elf2}, by simple modifications of the
arguments in \refS{sec:analysis}; hence \eqref{elf1} and \eqref{elf2} hold
for $s>0$.
Similarly, $E_s(k)/n$ is uniformly summable for any fixed $s>0$, and
\eqref{sumElim} follows.

There are no other changes to the proofs of Theorems \ref{thm.main} and \ref{thm.degrees}, since
$\mu=\sum kp_k$ does not appear in the proof (except in a trivial way before
\eqref{ebb}).  We thus have the following result for random multigraphs.

\begin{theorem} \label{thm.ht}
Let $(p_k)_0^{\infty}$ be a probability distribution. Assume that $n_k/n \to p_k$ for each $k \in \Z^+$ and that $\sumk k n_k/n$ converges to
a finite limit $\lambda$ as $n \to \infty$.
Let $S^{(n)}_\infty$ denote the size of a random greedy independent set in the random multigraph $G^*(n,(d_i)_1^n)$.

Let $\tau_\infty$ be the unique value in $(0,\infty]$ such that
\begin{equation*} 
\lambda \int_0^{\tau_\infty} \frac{e^{-2\sigma}}{\sum_k k p_k e^{-k\sigma}} \, d\sigma = 1.
\end{equation*}
Then
\begin{equation*} 
\frac{S^{(n)}_\infty}{n} \to \lambda \int_0^{\tau_\infty} e^{-2\sigma}
\frac{\sum_k p_k e^{-k\sigma}}{\sum_k k p_k e^{-k\sigma}} \, d\sigma \quad \mbox{in probability}.
\end{equation*}

Moreover, for each $k=0,1,\dots$,
\begin{equation*} 
\frac{S^{(n)}_\infty(k)}{n} \to \lambda \int_0^{\tau_\infty} e^{-2\sigma}
\frac{p_k e^{-k\sigma}}{\sum_j j p_j e^{-j\sigma}} \, d\sigma \quad \mbox{in
  probability}.
\end{equation*}
\end{theorem}

\begin{example}\label{Estar}
  For an extreme example, let $d_1=n-1$ and $d_2=\dots=d_n=1$, \ie, there is
  one vertex of degree $n-1$ and all others have degree 1. The limiting
  distribution is concentrated at 1: $p_k=\gd_{k1}$ so $\mu=1$, but the
 average vertex degree $\sum_k kn_k/n = 2(1-1/n)$ which converges to
 $\lambda=2>\mu$.

\refT{thm.main} applies. The equation \eqref{eq.tauinfty1} is
$\lambda\int_0^{\tau_\infty}e^{-\gs}\,d\gs=1$, and thus $\tau_\infty=\log 2$.
Then \eqref{eq.jamming1} yields, with convergence in probability,
\begin{equation}
\frac{S_\infty}n \to s_\infty
=\lambda\int_0^{\tau_\infty} e^{-2\gs}\,d\gs
=2\int_0^{\log2} e^{-2\gs}\,d\gs
=\frac{3}4.
\end{equation}

This can also be seen directly.
W.h.p., $G^*(n, (d_i)_1^n)$ consists of a star with centre at vertex 1,
$\approx n/2$ leaves and $\approx n/4$ loops at the centre, together with
$\approx n/4$ isolated edges. W.h.p., one of the leaves goes
off before the centre of the star, and then the centre is blocked and the
greedy independent set will actually be of maximum size, containing all the leaves of
the star and one vertex from each isolated edge, together $\approx 3n/4$
vertices.

Note that only the multigraph version of the theorem applies in this case.
The random simple graph $G(n, (d_i)_1^n)$ is deterministic and is a star
with $n-1$ leaves. W.h.p, the first vertex that goes off is a leaf; then the
centre is blocked and the greedy independent set will consist of all the leaves.
Thus $S_\infty=n-1$ \whp{} in the simple graph case.
\end{example}

Finally, we consider the case where $\lambda=\infty$.
In this case, our formulae
are no longer meaningful, but we can give at least a partial description of the outcome of the process via a different argument.
In the case where $\sum_k p_k = 1$, the result below implies that, \whp, almost all the vertices end up
in the independent set, i.e., the jamming constant $s_\infty=1$.

\begin{theorem} \label{thm.v-large}
Suppose that
$\gln = \sum_k k n_k /n \to\infty$.  Let $r(n)$ be the number of vertices
of degree at most $\min(\gln^{1/8}, n^{1/6})$.  Then, for every $\delta > 0$, $S_\infty \ge (1 - \delta)r(n)$ \whp
\end{theorem}

\begin{proof}
By considering subsequences, we may assume that $n_0/n\to p_0$ for some
$p_0\in[0,1]$. The case $p_0=1$ is trivial, so we assume
$p_0 = \lim_{n\to \infty} n_0/n < 1$.

Let $\eps = \eps(n) = \max( \gln^{-1/7}, n^{-1/5} ) \to 0$, and note that $\eps^6 \gln/ \log(1/\eps)$ and $\eps^4 n$ both tend to infinity.

Our first aim is to show that, at time $\eps$, \whp, there are at most $3\eps \gln n$ half-edges incident with empty vertices.
The total number of half-edges incident with empty vertices of degrees up to $\eps \gln$ is at most $\eps \gln n$, so we turn our
attention to vertices of larger degree.

Let $\tau_1$ be the first time that there are fewer than $\eps^3 n$ empty vertices of degree
at least $\eps \gln$.  Let $\tau_2$ be the time when the $\lceil \frac14 \eps^5 \gln n \rceil$-th edge is paired.
We claim that (a)~$\tau_1 \le \tau_2$ \whp, and (b)~$\tau_1 \wedge \tau_2
\le \eps$ w.h.p.
Together these two imply that $\tau_1 \le \eps$ \whp

To show (a), note that, as in the proof of Lemma~\ref{LEE}, for every vertex $v$ of degree $k \ge \eps \gln$,
\begin{equation*} 
\Pr \big( v \mbox{ is empty at time } \tau_2 \big) \le \Big( 1 - \frac{k}{\gln n} \Big)^{\frac14 \eps^5 \gln n}
\le e^{- \frac14 \eps^6 \gln} = o(\eps^3).
\end{equation*}
Hence the expected number of empty vertices of degree at least $\eps \gln$ at time $\tau_2$ is
$o(\eps^3 n)$, and the probability that there are as many as $\eps^3 n$ such vertices is $o(1)$.

To show (b), note that, before time $\tau_1$, vertices of degree at least $\eps \gln$ go off at a rate of at
least $\eps^3 n$; when any such vertex goes off, at least $\frac12 \eps \gln$ edges are paired.
For $0 \le t \le \tau_1$, let $P_t$ be the number of edges paired by time~$t$.
Then there is a Poisson process $Q_t$ of intensity $\eps^3 n$ such that $P_t / (\frac12 \eps \gln)$
dominates $Q_t$ for $t \le \tau_1$, and
\begin{equation*} 
\Pr \big(Q_\eps < \frac12 \eps^4 n \big) = o(1).
\end{equation*}
This implies that, \whp, either $\tau_1 < \eps$ or $P_\eps \ge \frac14 \eps^5 \gln n$, in which case $\tau_2 \le \eps$.
This completes the proof of~(b).

We conclude that, \whp, $\tau_1 \le \eps$, so that the number of empty
vertices of degree at least
$\eps \gln$ at time $\eps$ is at most
$\eps^3 n$.  It follows that the number of half-edges incident with empty vertices of degrees between $\eps \gln$ and $\eps^{-2} \gln$
at time $\eps$ is at most $\eps \gln n$.

It remains to consider empty vertices of degrees greater than $\eps^{-2} \gln$.
For this range, we repeat the argument in Lemma~\ref{LEE}.
We take $n$ large enough that there are at least $\frac12 (1-p_0)n$ vertices of positive degree.
Let $\tau_\eps$ be the time when the $\lceil \frac19 \eps (1-p_0) n\rceil$-th edge is created by a pairing.
Until this time, there are always at least $\frac14 (1-p_0) n$ empty vertices of positive degree, and so the pairing rate is at least $\frac14 (1-p_0)n$.
Thus
\begin{equation*} 
\Pr(\tau_\eps > \eps) \le \Pr \Big( \Po \big(\frac14 \eps (1-p_0) n \big) < \frac18 \eps (1-p_0) n\Big) = o(1).
\end{equation*}
Now, for each vertex $v$ of degree $k \ge \eps^{-2}\gln$,
\begin{equation*} 
\Pr \big( v \mbox{ is empty at time } \tau_\eps \big) \le \Big( 1 - \frac{k}{\gln n} \Big)^{\frac19 \eps (1-p_0) n}
\le e^{- \frac19 \eps^{-1} (1-p_0)} = o(\eps).
\end{equation*}
Thus the expected number of half-edges incident with empty vertices of
degree at least $\eps^{-2} \gln$ at time $\tau_\eps$ is
$o(\eps \gln n)$.  Hence, \whp, there are at most $\eps \gln n$ such
half-edges at time $\eps$.

Summing over the three ranges of degrees, we see that
the number $U$ of half-edges incident with empty vertices at time $\eps$ is \whp{} at most
$3 \eps \gln n$.
The expected total number of half-edges incident with vertices that have gone off by time~$\eps$ is at most $\eps \gln n$, so \whp{}
at most $\frac 12 \gln n$ of the half-edges are paired by this time.  Hence, \whp, there are at least $\frac12 \gln n$ free half-edges at time~$\eps$,
most of which are incident with vertices that have already become blocked.

Now consider any vertex $v$ of degree $k \le k_0$, where
$k_0 = k_0(n) := \min( \gln^{1/8}, n^{1/6}) = o(\eps^{-1})$.  The
probability that $v$ has gone off by time $\eps$ is at most $\eps$
and the probability that $v$ is blocked by time $\eps$ is at most
$k_0\eps=o(1)$.

The conditional probability that $v$ will be blocked after time $\eps$, given
that $v$ has neither gone off nor has been blocked by time
$\eps$, and also given that $U\le3\eps\gln n$ and there are at least $\frac 12 \gln n$ free half-edges at time $\eps$,
is at most
the conditional probability that
any of the half-edges incident with $v$
are later paired with
a half-edge incident to another vertex that is empty at time $\eps$,
which is at most $k_0 U / \frac12 \gln n\le 6\eps k_0=o(1)$.
Thus the (unconditional) probability that $v$ will be blocked is
$o(1)$ and $v$ ends up in the independent set with probability
$1-o(1)$.  Hence the expected number of
vertices of degree at most $k_0$ that do not appear in $S_\infty$ is
$o(r(n))$, and the result follows.
\end{proof}

The function $\min (\gln^{1/8},n^{1/6})$ can certainly be improved via an adjustment of the parameters in the argument, but it seems unlikely that
this will give the best possible result.  As we shall see in Example~\ref{ex.F} below, some dependence on $n$ is essential.

We finish with two examples to illustrate some of the possible behaviours in this range.

\begin{example}
Fix constants $\alpha, \gamma>0$ with $2/3 > \alpha > 1/2 + \gamma$.

Take a set $A$ of $n^\alpha$ vertices of degree $n^\alpha$,
and a set $B$ of $n-n^\alpha$ vertices of degree $n^\gamma$.
So all but a proportion $\approx n^{1+\gamma -2\alpha} \ll n^{-\gamma}$ of
the half-edges are incident with $A$,
and so most vertices of $B$ have all their neighbours in $A$.  We have $\gln(n) \approx n^{2\alpha -1}$, and
the degrees of vertices in $B$ can be up to $\gln^{1/2 -\delta}$ for any $\delta > 0$.

Now consider what happens when the first $n^\alpha$ vertices go off
(including vertices that are already blocked when they go off).
W.h.p., at most $2n^{2\alpha -1}$ of the vertices that go off among the first $n^\alpha$ are in $A$,
and so at most $2n^{3\alpha -1} = o(n)$ vertices in $B$ become blocked as a result.
Also at most $n^{\alpha + \gamma} =o(n)$ vertices in $B$ become blocked as a result of vertices in $B$ going off.
Meanwhile, \whp, at least $\frac12 n^\alpha$ vertices in $B$ go off, and
this generates at least $\frac14 n^{\alpha + \gamma}$ pairings,
most of which are to vertices in $A$.  The probability that a given vertex of $A$ is not blocked during this process is
thus exponentially small.

So at this point, \whp, all vertices of $A$ are blocked, and almost all vertices of $B$ are empty,
with all their neighbours blocked.  Hence the independent set has size $(1-o(1))|B| = (1-o(1))n$.

This example shows, among other things, that the jamming constant of the random multigraph can approach~1 even if
all the $p_i$ are equal to~0.
\end{example}

If we allow $\gln$ to be extremely large, then there is an even greater variety of possible behaviour.
In the example below, the jamming constant is very far from being concentrated.

\begin{example} \label{ex.F}
Let $A$ be a set of $n/2$ vertices of degree $n^{1+\delta}$, for some $\delta > 0$, and let $B$ be a set of $n/2$ vertices of degree $n^\beta$,
where $\beta \ge 3 + 5\delta$.  So $\gln$ is approximately $\frac12 n^\beta$.  Then, \whp, every vertex in $B$ is adjacent to every vertex of
$A\cup B$, and no two vertices of $A$ are adjacent.

In this instance, if a vertex of $B$ is the first to go off, then $S_\infty = 1$, while if a vertex of $A$ is the first to go off, then
$S_\infty = n/2$.

By making $\beta$ large, we can ensure that the degrees of vertices of $A$ can be an arbitrarily small power of $\gln$.
\end{example}

\newcommand\vol{\textbf}
\newcommand\jour{\emph}
\newcommand\book{\emph}
\newcommand\inbook{\emph}
\def\no#1#2,{\unskip#2, no. #1,} 
\newcommand\toappear{\unskip, to appear}

\newcommand\arxiv[1]{\texttt{arXiv:#1}}
\newcommand\arXiv{\arxiv}

\def\nobibitem#1\par{}

\end{document}